\documentclass[12pt, reqno]{amsart}
\usepackage{amsmath, amsthm, amscd, amsfonts, amssymb, graphicx, color}
\usepackage[bookmarksnumbered, colorlinks, plainpages]{hyperref}
\hypersetup{colorlinks=true,linkcolor=red, anchorcolor=green, citecolor=cyan, urlcolor=red, filecolor=magenta, pdftoolbar=true}

\newtheorem{theorem}{Theorem}[section]
\newtheorem{lemma}[theorem]{Lemma}

\theoremstyle{definition}
\newtheorem{definition}[theorem]{Definition}

\theoremstyle{remark}
\newtheorem{remark}[theorem]{Remark}
\numberwithin{equation}{section}

\begin{document}

\setcounter{page}{1}

\title[Inhomogeneous pseudo-parabolic equation ..]{Nonexistence of global solutions for an inhomogeneous pseudo-parabolic equation}

\author[M. B. Borikhanov and B. T. Torebek]{Meiirkhan B. Borikhanov and  Berikbol T. Torebek}

\address{{Meiirkhan B. Borikhanov \newline Khoja Akhmet Yassawi International Kazakh--Turkish University \newline Sattarkhanov ave., 29, 161200 Turkistan, Kazakhstan \newline Department of Mathematics: Analysis, Logic and Discrete Mathematics \newline
Ghent University, Belgium}}
\email{meiirkhan.borikhanov@ayu.edu.kz, meiirkhan.borikhanov@ugent.be}

\address{{Berikbol T. Torebek \newline Institute of
Mathematics and Mathematical Modeling \newline 125 Pushkin str.,
050010 Almaty, Kazakhstan \newline Department of Mathematics: Analysis, Logic and Discrete Mathematics \newline
Ghent University, Belgium}}
\email{{torebek@math.kz, berikbol.torebek@ugent.be}}


\let\thefootnote\relax\footnote{$^{*}$Corresponding author}



\begin{abstract}In the present paper, we study an inhomogeneous pseudo-parabolic equation with nonlocal nonlinearity
$$u_t-k\Delta u_t-\Delta u=I^\gamma_{0+}(|u|^{p})+\omega(x),\,\ (t,x)\in(0,\infty)\times\mathbb{R}^N,$$
where $p>1,\,k\geq 0$, $\omega(x)\neq0$ and $I^\gamma_{0+}$ is the left Riemann-Liouville fractional integral of order $\gamma\in(0,1).$ Based on the test function method, we have proved the blow-up result for the critical case $\gamma=0,\,p=p_c$ for $N\geq3$, which answers an {\bf open question} posed in \cite{Zhou}, and in particular when $k=0$ it improves the result obtained in \cite{Bandle}. An interesting fact is that in the case $\gamma>0$, the problem does not admit global solutions  for any $p>1$ and $\int_{\mathbb{R}^N}\omega(x) dx>0.$\end{abstract}
\maketitle

\section{Introduction}
Recently, Zhou in \cite{Zhou} has investigated the inhomogeneous pseudo-parabolic equation
\begin{equation}\label{03}
\left\{\begin{array}{l}\large\displaystyle
u_t-k\Delta u_t-\Delta u=|u|^{p}+\omega(x),\,\ (t,x)\in(0,\infty)\times\mathbb{R}^N,\\
u(0,x)=u_0(x),\, x\in\mathbb{R}^N,\end{array}\right.\end{equation}
where $p>0$, $k>0$ and $u_0,\,\omega\in C_0(\mathbb{R}^N).$

There was studied the effect of the inhomogeneous term $\omega(x)$ on the critical exponent $p_c$ of the problem \eqref{03}, and it was proven that for
\begin{equation*}\label{0.4} p_c=\left\{\begin{array}{l}
\infty\,\,\,\text{if}\,\,\,N=1,2, \\
\large\displaystyle\frac{N}{N-2}\,\,\,\text{if}\,\,\,N\geq3, \end{array}\right.\end{equation*}
(a) if $1<p<p_c,\,u_0\geq0$ and $\int_{\mathbb{R}^N}\omega(x) dx>0$,  then the solution of \eqref{03} blows up in finite time.\\
(b) if $p> p_c$, then there exist $u_0\geq0$ and $\omega\geq0$ such that the problem \eqref{03} admits global solutions.

Note that the critical case $p=p_c$ was left open (see \cite[Remark 4(b)]{Zhou}).

At first, the problem \eqref{03} for $\omega(x)\equiv0$ has studied in \cite{Cao, Kaikina}.
It is shown that there exists the critical exponent $p_F=1+\frac{2}{N}$, for the pseudo-parabolic equation. This exponent coincides with the Fujita critical exponent of the semilinear heat equations, which was first introduced by Fujita in \cite{Fujita1}.

The problem \eqref{03} with $k = 0$ is considered by Bandle et. al. \cite{Bandle}. Namely, it was studied the cases $(a), (b)$ and \\
(c) if $N\geq3$, $p={p}_{c}$, $\large\displaystyle\int_{\mathbb{R}^N}\omega(x)dx >0$, $\omega(x)={O}(|x|^{-\varepsilon-N})$ as $|x|\to\infty$ for some $\varepsilon>0$, and either $u\geq0$ or
$$\int_{|x|>R}\frac{\omega^-(y)}{|x-y|^{N-2}}dy=\frac{o(1)}{|x|^{N-2}},\, \omega^-=\max\{-\omega, 0\}$$
when $R$ is enough large, then \eqref{03} has no global solutions.

Later on, Jleli et. al. \cite{Jleli}  generalized these results with the forcing term $t^\sigma\omega(x),$ $\sigma>-1$, and showed the effects of forcing term on the critical exponents.

In this paper, we study the semilinear pseudo-parabolic equation with a forcing term depending on the space
\begin{equation}\label{01}
\left\{\begin{array}{l}\large\displaystyle
u_t-k\Delta u_t-\Delta u=I^\gamma_{0+}(|u|^{p})+\omega(x),\,\ (t,x)\in(0,\infty)\times\mathbb{R}^N,\\
u(0,x)=u_0(x),\, x\in\mathbb{R}^N,\end{array}\right.\end{equation}
where $p>1,\,k\geq 0$, $\omega(x)\neq0$ and $I^\gamma_{0+}$ is the left Riemann-Liouville fractional integral of order $\gamma\in[0,1).$

We note that the problem \eqref{01} for $k=0$ and $\omega(x)\equiv0$, was considered in  \cite{BT, Cazenave, Fino, Sun}.

The main purpose of this paper is to prove a blow-up result for the critical case $p=p_c$ for $N\geq3$, thereby answering the {\it open question} proposed in \cite{Zhou}. In addition, to study the effect of nonlocal nonlinearity in time on the critical exponent.

\subsection{Preliminaries}
\begin{definition}[\cite{Kilbas}, p. 69]\label{RD} The left and right Riemann-Liouville fractional integrals of order $\gamma \in(0,1)$ for an integrable function $u(t),\,\,t\in (0,T)$ are given by
$$I^\gamma_{0+}u(t)=\int\limits_{0}^{t}{{\frac{{\left( t-s \right)}^{\gamma -1}}{\Gamma \left( \gamma  \right)}}}u\left( s \right)ds$$ and $$I^\gamma_{T-}u(t)=\int\limits_{t}^{T}{{\frac{{\left( s-t \right)}^{\gamma -1}}{\Gamma \left( \gamma  \right)}}}u\left( s \right)ds.$$
Since $I^\gamma [u](t)\to u(t)$ almost everywhere as $\gamma\to0$ (see \cite{Kilbas}), we can let $I^0 [u](t)= u(t)$.

\end{definition}
\begin{definition}[Weak solution]\label{WS} We say that $u\in L^p_\text{loc}([0,\infty)\times\mathbb{R}^N)$ is a global weak solution to \eqref{01}, if
\begin{equation}\label{W01}\begin{split}
\int_0^T\int_{\mathbb{R}^{N}}&|u|^{p}(I^\gamma_{T-}\varphi)  dx dt +\int_0^T\int_{\mathbb{R}^{N}} \omega\varphi dx dt+\int_{\mathbb{R}^{N}}u_0(\varphi(0,x)-k\Delta\varphi(0,x))dx
\\&= -\int_0^T\int_{\mathbb{R}^{N}}u\varphi_tdx dt+k\int_0^T\int_{\mathbb{R}^{N}}u\Delta\varphi_tdx dt-\int_0^T\int_{\mathbb{R}^{N}}u\Delta\varphi dx dt,
\end{split}\end{equation}
holds for all $T>0$ and $\varphi\in C^{1,2}_{t,x}([0,T], \mathbb{R}^N), \varphi\geq0,$ $\text{supp}_x\varphi\subset\subset\mathbb{R}^N$ and $\varphi(T,\cdot)=0.$
\end{definition}
\begin{lemma}\label{TF}\cite[Lemma 3.1]{Pokhozaev1} Let $\omega\in L^1(\mathbb{R}^N)$ and $\large\displaystyle\int_{\mathbb{R}^N}\omega(x) dx>0.$ Then there exists a test function $0\leq\phi\leq1$
such that
$\int_{\mathbb{R}^N}\omega\phi dx>0.$
\end{lemma}

\section{Main results}  In this section, we will show the blow-up of the solution to \eqref{01} using the test function method.
\begin{theorem}\label{MM}
Let $u_0,\,\omega\in L^1(\mathbb{R}^N)$ and $\int_{\mathbb{R}^N}\omega(x) dx>0.$
Then
\\{\bf (i)} if $\gamma>0$, then for any $p>1$ the problem \eqref{01} admits no global weak solution.
\\{\bf (ii)} if $\gamma=0$ and $p=p_c=\frac{N}{N-2},\,N\geq3$, then the problem \eqref{01} admits no global weak solution.
\end{theorem}
\begin{remark}
Note that the part (ii) of Theorem \ref{MM} answers to the open question posed by Zhou in \cite{Zhou}.
\end{remark}
\begin{remark}
When $k=0$ the equation \eqref{01} coincides with the heat equation considered in \cite{Bandle}, then our results remain true for the heat equation. Note that the part (i) of Theorem \ref{MM} in the case $k=0$ improves the result in \cite{Samet}, since we do not assume that $u_0$ is positive. The part (ii) of Theorem \ref{MM}, in case $k=0$ improves the result in \cite{Bandle}. Since we do not assume some asymptotic properties of the function $\omega(x)$ as in \cite{Bandle}, our result improves part (b) of Theorem 2.1 from \cite{Bandle}.
\end{remark}
\begin{proof}[Proof of Theorem \ref{MM}.]
We present the proofs of the cases {\bf (i)} and {\bf (ii)} separately.

{\bf (i)} \textbf{The case $\gamma>0$ and $p>1$.} The proof is done by contradiction.
\\Assume that $u$ is a global weak solution to problem \eqref{01}. We choose the test function in the following form $$\varphi(t,x)=\psi(t)\xi(x),$$
with
\begin{equation*}\label{QT2}
\psi(t)=\biggl(1-\frac{t}{T}\biggr)^m,\,m>\frac{p+\gamma}{p-1},\,\,\, t\in[0,T],\,\,\,T\in(0,\infty),  \end{equation*}
and
\begin{equation*}\label{QT1}
\xi(x)=\Phi\left(\frac{|x|^2}{R^2}\right),\,\,R\gg1,\,\, x\in\mathbb{R}^{N}.
\end{equation*}

Let $\Phi(z)\in C_0^\infty(\mathbb{R_+})$ be a nonincreasing function
\begin{equation*}
\Phi(z)=
 \begin{cases}
   1 &\text{if $0\leq z\leq1 $},\\
   \searrow &\text{if $1<z<2$},\\
   0 &\text{if $z\geq 2$}.
 \end{cases}
\end{equation*}
Then, from \eqref{W01} it follows that
\begin{equation}\label{W02}\begin{split}
\int_0^T&\int_{\mathbb{R}^{N}}|u|^{p}(I^\gamma_{T-}\varphi)  dx dt +\int_0^T\int_{\mathbb{R}^{N}} \omega\varphi dx dt+\int_{\mathbb{R}^{N}}u_0(\varphi(0,x)-k\Delta\varphi(0,x))dx
\\&\leq \int_0^T\int_{\mathbb{R}^{N}}|u||\varphi_t|dx dt+k\int_0^T\int_{\mathbb{R}^{N}}|u||\Delta\varphi_t|dx dt+\int_0^T\int_{\mathbb{R}^{N}}|u||\Delta\varphi| dx dt.
\end{split}\end{equation}
Using the $\varepsilon$-Young inequality in the right-side of \eqref{W02} with $\large\displaystyle\varepsilon=\frac{p}{3}$, we obtain
\begin{equation*}\label{M1}\begin{split}
\int_0^T\int_{\mathbb{R}^{N}}|u|\left|\varphi_t\right|dx dt&\leq\frac{1}{3}\int_0^T\int_{\mathbb{R}^{N}}|u|^{p}(I^\gamma_{T-}\varphi)  dx dt
\\&+\frac{p-1}{p}\biggl(\frac{p}{3}\biggr)^{-\frac{1}{p-1}}\underbrace{\int_0^T\int_{\mathbb{R}^{N}}(I^\gamma_{T-}\varphi)^{-\frac{1}{p-1}}\left|\varphi_t\right|^{\frac{p}{p-1}}dx dt}_{\mathcal{I}_1}. \end{split}\end{equation*}
Similarly, one obtains
\begin{equation*}\label{M2}\begin{split}
\int_0^T\int_{\mathbb{R}^{N}}|u|\left|\Delta\varphi_t\right|dx dt&\leq\frac{1}{3}\int_0^T\int_{\mathbb{R}^{N}}|u|^{p}(I^\gamma_{T-}\varphi)  dx dt
\\&+\frac{p-1}{p}\biggl(\frac{p}{3}\biggr)^{-\frac{1}{p-1}}\underbrace{\int_0^T\int_{\mathbb{R}^{N}}(I^\gamma_{T-}\varphi)^{-\frac{1}{p-1}}\left|\Delta\varphi_t\right|^{\frac{p}{p-1}}dx dt}_{\mathcal{I}_2},
\end{split}\end{equation*}

and
\begin{equation*}\label{M3}\begin{split}
\int_0^T\int_{\mathbb{R}^{N}}|u|\left|\Delta\varphi\right|dx dt&\leq\frac{1}{3}\int_0^T\int_{\mathbb{R}^{N}}|u|^{p}(I^\gamma_{T-}\varphi)  dx dt
\\&+\frac{p-1}{p}\biggl(\frac{p}{3}\biggr)^{-\frac{1}{p-1}}\underbrace{\int_0^T\int_{\mathbb{R}^{N}}(I^\gamma_{T-}\varphi)^{-\frac{1}{p-1}}\left|\Delta\varphi\right|^{\frac{p}{p-1}}dx dt}_{\mathcal{I}_3}.\end{split}\end{equation*}
Therefore, we can rewrite the inequality \eqref{W02} in the following form
\begin{equation}\label{W03}\begin{split}
&\int_0^T\int_{\mathbb{R}^{N}} \omega\varphi dx dt+\int_{\mathbb{R}^{N}}u_0(\varphi(0,x)-k\Delta\varphi(0,x))dx
\leq C(p)\biggr(\mathcal{I}_1+k\mathcal{I}_2+\mathcal{I}_3\biggr),
\end{split}\end{equation} where $C(p)=\frac{p-1}{p}\biggl(\frac{p}{3}\biggr)^{-\frac{1}{p-1}}.$

Next, we estimate the integrals $\mathcal{I}_1, \mathcal{I}_2, \mathcal{I}_3$. At this stage, inserting the equality
\begin{equation}\label{M4}\begin{split}
(I^\gamma_{T-}\psi)(t)=\frac{\Gamma(m+1)}{\Gamma(\gamma+m+1)}T^{\gamma}\biggl(1-\frac{t}{T}\biggr)^{m+\gamma},\,\,\, t\in[0,T) ,\end{split}\end{equation}to the term of the above integrals and changing the variable $y=xR$, we obtain
\begin{equation}\label{M5}\begin{split}
&\mathcal{I}_1\leq CT^{\frac{-\gamma-1}{p-1}}R^N,
\\&\mathcal{I}_2\leq CT^{\frac{-\gamma-1}{p-1}}R^{N-\frac{2p}{p-1}},
\\&\mathcal{I}_3\leq CT^{1-\frac{\gamma}{p-1}}R^{N-\frac{2p}{p-1}}.\end{split}\end{equation}
On the other hand, it follows from a simple calculation that
\begin{equation}\label{M6}\begin{split}
\int_0^T\psi(t)dt=\int_0^T\biggl(1-\frac{t}{T}\biggr)^mdt=C(m)T.\end{split}\end{equation}
Combining  \eqref{W03}-\eqref{M6} we arrive at
\begin{equation*}\label{W04}\begin{split}
\int_{\mathbb{R}^{N}} \omega\xi dx&+C(m)T^{-1}\int_{\mathbb{R}^{N}}u_0(\xi-k\Delta\xi)dx \\&\leq C(p,m)\biggr(CT^{\frac{-\gamma-1}{p-1}-1}R^N+ kCT^{\frac{-\gamma-1}{p-1}-1}R^{N-\frac{2p}{p-1}}+ CT^{-\frac{\gamma}{p-1}}R^{N-\frac{2p}{p-1}}\biggr).
\end{split}\end{equation*}
Finally, fixing $R$ and passing $T\to+\infty$ in the last inequality and using Lemma \ref{TF}, we deduce that
$\int_{\mathbb{R}^{N}} \omega\xi dx \leq 0,$ which is a contradiction.

{\bf (ii)} \textbf{The critical case $\gamma=0$ and $p=p_c=\frac{N}{N-2},\,N\geq 3.$} The proof also will be done by contradiction. Suppose that $u$ is a global weak solution to \eqref{01}.
\\Now, following the idea of \cite{Agarwal}, we set the test function as $$\varphi(t,x)=\eta(t)\phi(x),$$
for large enough $R, T$
\begin{equation}\label{T2}
\eta(t)=\nu\left(\frac{t}{T}\right),\,\,\, t>0,    \end{equation}
and
\begin{equation}\label{T1}
\phi(x)=\Psi\left(\frac{\ln\left(\frac{|x|}{\sqrt{R}}\right)}{\ln\left(\sqrt{R}\right)}\right),\,\, x\in\mathbb{R}^{N}.
\end{equation}
Let $\nu\in C^\infty(\mathbb{R})$ be such that
$\nu\geq0;\,\,\,\nu\not\equiv0;\,\,\,\text{supp}(\nu)\subset(0,1),$ and $\Psi:\mathbb{R}\to[0,1]$ be a smooth function  satisfying
\begin{equation}\label{TT}
\Psi(s)=\left\{\begin{array}{l}
1,\,\,\text{if}\,\,-\infty< s\leq0,\\
0,\,\,\text{if}\,\,s\geq1.\end{array}\right.
\end{equation}
and there exist positive constants $\theta_1, \theta_2$ such that
\begin{equation}\label{QWQ}
|\phi''(x)|\leq\theta_1|\phi(x)|,\,\,|\phi'(x)|\leq\theta_2|\phi(x)|. \end{equation}
Using the fact that  $\text{supp}(\nu)\subset(0,1)$, we can easily get
\begin{equation}\label{support}
\int_{\mathbb{R}^{N}}u_0(\varphi(0,x)-k\Delta\varphi(0,x))dx=\nu(0)\int_{\mathbb{R}^{N}}u_0(\phi(x)-k\Delta\phi(x))dx=0.
\end{equation}
Then, acting in the same way as in the above case, we get the following estimate
\begin{equation}\label{T3}\begin{split}
&\int_0^T\int_{\mathbb{R}^{N}} \omega\varphi dx dt\leq C(p)\biggr(\mathcal{J}_1+k\mathcal{J}_2+\mathcal{J}_3\biggr),
\end{split}\end{equation} with \begin{equation*}\begin{split}
&\mathcal{J}_1=\int_0^T\int_{\mathbb{R}^{N}}\varphi^{-\frac{1}{p-1}}\left|\varphi_t\right|^{\frac{p}{p-1}}dx dt,
\\&\mathcal{J}_2=\int_0^T\int_{\mathbb{R}^{N}}\varphi^{-\frac{1}{p-1}}\left|\Delta\varphi_t\right|^{\frac{p}{p-1}}dx dt,
\\&\mathcal{J}_3=\int_0^T\int_{\mathbb{R}^{N}}\varphi^{-\frac{1}{p-1}}\left|\Delta\varphi\right|^{\frac{p}{p-1}}dx dt.
\end{split}\end{equation*}
In view of \eqref{T2} and \eqref{T1}, let us calculate the next integral
\begin{equation}\begin{split}\label{J2}
\mathcal{J}_2=\biggl( \int_0^T\eta^{-\frac{1}{p-1}}\left|\eta_t\right|^{\frac{p}{p-1}}dt\biggr)\biggl(\int_{\mathbb{R}^{N}}\phi^{-\frac{1}{p-1}}\left|\Delta\phi\right|^{\frac{p}{p-1}}dx\biggr).
\end{split}\end{equation}
Indeed, the function $\phi$ is a radial, and remaining \eqref{QWQ} we arrive at
\begin{equation*}\begin{split}
|\Delta\phi|&=\frac{d^2\phi}{dr^2}+\frac{N-1}{r}\frac{d\phi}{dr}
=\phi''\frac{1}{r^2\ln^2\sqrt{R}}+\phi'\frac{N-2}{r^2\ln\sqrt{R}}\\&\leq \theta_1|\phi|\frac{1}{r^2\ln^2\sqrt{R}}+\theta_2|\phi|\frac{N-2}{r^2\ln\sqrt{R}}\\&\leq \frac{C}{r^2\ln R}|\phi|,
\end{split}\end{equation*}where $r=|x|=(x_1^2+x_2^2+...+x_n^2)^\frac{1}{2}.$\\
Since, $\large\displaystyle p=\frac{N}{N-2}$, by inserting the last inequality into \eqref{J2}, we can verify that
\begin{equation*}\begin{split}
\int_{\mathbb{R}^{N}}\phi^{-\frac{1}{p-1}}\left|\Delta\phi\right|^{\frac{p}{p-1}}dx\leq C^{\frac{N}{2}}(\ln R)^{-\frac{N}{2}}\int_{\mathbb{R}^{N}}\frac{|\phi|}{|x|^N}dx.\end{split}\end{equation*}
Using \eqref{T1}  and \eqref{TT}, we get
\begin{equation}\label{J21}\int_{\mathbb{R}^N}\phi^{-\frac{1}{p-1}}\left|\Delta\phi\right|^{\frac{p}{p-1}}dx\leq C(\ln R)^{\frac{2-N}{2}}.\end{equation}
Similarly, from \eqref{T2} one obtains
\begin{equation}\begin{split}\label{J22}
\int_0^T\eta^{-\frac{1}{p-1}}\left|\eta_t\right|^{\frac{p}{p-1}}dt=CT^{-\frac{N}{2}+1}. \end{split}\end{equation}
By combining \eqref{J21}-\eqref{J22}, we can rewrite \eqref{J2} as
\begin{equation*}\label{E2}
\mathcal{J}_2\leq kCT^{-\frac{N}{2}+1}(\ln R)^{\frac{2-N}{2}}.
\end{equation*}
Consequently, we will estimate the integrals $\mathcal{J}_1$ and $\mathcal{J}_3$, respectively, in the following form
$$\mathcal{J}_1\leq CT^{-\frac{N}{2}+1}R^N$$ and $$\mathcal{J}_3\leq CT^{1}(\ln R)^\frac{2-N}{2}.$$
Finally, we deduce that
\begin{equation}\label{ILQ}\begin{split}
&\int_{\mathbb{R}^{N}} \omega\phi dx \leq C(p)\biggr(CT^{-\frac{N}{2}}R^N+kCT^{-\frac{N}{2}}(\ln R)^\frac{2-N}{2}+C(\ln R)^\frac{2-N}{2}\biggr).
\end{split}\end{equation}
Now for $T=R^j, j>0,$ we get
\begin{equation*}\begin{split}
\int_{\mathbb{R}^{N}} \omega\phi dx\leq C\left(CR^{-\frac{N(j-2)}{2}}+kCR^{-\frac{N(j-2)}{2}}(\ln R)^{\frac{2-N}{2}}+C(\ln R)^{\frac{2-N}{2}}\right). \end{split}\end{equation*}
Taking $j>2$ and passing to the limit as $R\to\infty$ in the above inequality and in view of Lemma \ref{TF}, we deduce that $\int_{\mathbb{R}^{N}} \omega\phi dx\leq0,$ which is a contradiction.\end{proof}

\section*{\large Acknowledgments}
This research has been funded by the Science Committee of the Ministry of Education and Science of the Republic of Kazakhstan (Grant No. AP09259578) and by the FWO Odysseus 1 grant G.0H94.18N: Analysis and Partial Differential Equations. No new data was collected or generated during the course of research.

\end{document}